\documentclass[12pt]{amsart}
\usepackage{amsmath}
\usepackage{amssymb}
\usepackage{latexsym}

\newcommand{\Zint}{\mathbb {Z}}    
     
\newcommand{\R}{\mathbb {R}}      
\newcommand{\halmos}{\rule{5pt}{5pt}}

\numberwithin{equation}{section}

\newtheorem{prop}{\bf Proposition}[section]
\newtheorem{thm}[prop]{\bf Theorem}

\theoremstyle{definition}

\begin{document}
\title[Ultradiscrete limit of the spectral polynomial]
{Ultradiscrete limit of the spectral polynomial of the $q$-Heun equation}
\author{Kentaro Kojima}
\author{Tsukasa Sato}
\author{Kouichi Takemura}
\address{Department of Mathematics, Faculty of Science and Engineering, Chuo University, 1-13-27 Kasuga, Bunkyo-ku Tokyo 112-8551, Japan}
\email{takemura@math.chuo-u.ac.jp}
\begin{abstract}
It is known that the $q$-Heun equation has polynomial-type solutions in some special cases, and the condition for the accessory parameter $E$ is described by the roots of the spectral polynomial. 
We investigate the spectral polynomial by considering the ultradiscrete limit.
\end{abstract}
\maketitle

\section{Introduction}
A $q$-difference analogue of Heun's differential equation was introduced by Hahn \cite{Hahn} in the form
\begin{equation}
a(x) g(x/q) + b(x) g(x) + c(x) g(qx) =0
\label{eq:axgbxgcxg}
\end{equation}
such that $a(x)$, $b(x)$, $c(x)$ are polynomials such that $\deg_x a(x)= \deg_x c(x)=2 $, $a(0) \neq 0 \neq c(0) $ and $\deg _x b(x) \leq 2$.
It was rediscovered in \cite{TakR} by the fourth degeneration of Ruijsenaars-van Diejen system (\cite{vD0,RuiN}) or by specialization of the linear difference equation associated to the $q$-Painlev\'e VI equation (\cite{JS}).
We adopt the expression of the $q$-Heun equation as
\begin{align}
& (x-q^{h_1 +1/2} t_1) (x- q^{h_2 +1/2} t_2) g(x/q)  \label{eq:RuijD5} \\
& + q^{\alpha _1 +\alpha _2} (x - q^{l_1-1/2}t_1 ) (x - q^{l_2 -1/2} t_2) g(qx) \nonumber \\
& -\{ (q^{\alpha _1} +q^{\alpha _2} ) x^2 + E x + q^{(h_1 +h_2 + l_1 + l_2 +\alpha _1 +\alpha _2 )/2 } ( q^{\beta /2}+ q^{-\beta/2}) t_1 t_2 \} g(x) =0, \nonumber
\end{align}
which was employed in \cite{TakqH,KST}.
The parameter $E$ is called the accessory parameter.
Note that we recover Heun's differential equation by the limit $q\to 1 $ (\cite{Hahn,TakR}).
Recently, the $q$-Heun equation appears in the study of degenerations of the Askey-Wilson algebra (\cite{BVZ}).

Solutions of the $q$-Heun equation were considered in \cite{TakqH,KST}.
We investigate a solution of the $q$-Heun equation written as
\begin{equation}
f(x)= x^{\lambda _1} \sum _{n=0}^{\infty } c_n (E)  x^n , \; \lambda _1 = \frac{h_1 +h_2 -l_1-l_2 -\alpha _1-\alpha _2 -\beta +2}{2}.
\label{eq:A4la1la2}
\end{equation}
Note that the value $\lambda _1 $ is one of the exponents of $q$-Heun equation about $x=0$.
Then the coefficients $c_{n} (E)$ $(n=1,2,\dots )$ are determined recursive by 
\begin{align}
& c_n (E) t_1 t_2 [ q^{h_1 +h_2 } ( 1 - q^{n }) ( 1 - q^{n -\beta }) ] \label{eq:rec01}  \\
& = c_{n-1} (E)  [E q^{n -1 +\lambda _1} + q^{1/2}(q^{h_1 } t_1 +q^{h_2 } t_2 ) +  (q^{l_1 } t_1 +q^{l_2 } t_2 ) q^{2(n +\lambda _1) +\alpha _1 +\alpha _2 -5/2 }]  \nonumber \\
& - c_{n-2} (E) [ q (1 - q^{n -2 +\lambda _1 +\alpha _1})(1 - q^{n-2 +\lambda _1 + \alpha _2}) ] , \nonumber 
\end{align}
with the initial condition $c_0(E)=1$ and $c_{-1} (E)=0$ (see \cite{KST}).
If we regard $E$ as an indeterminant, then $c_n (E)$ is a polynomial of $E $ of degree $n$.
The polynomial type solution of the $q$-Heun equation, which is written as a terminating series, is described as follows.
\begin{prop} (\cite{KST}) \label{prop:prop0}
Let $\lambda _1$ be the value in Eq.(\ref{eq:A4la1la2}) and assume that $-\lambda _1 - \alpha _1 (:=N)$ is a non-negative integer and $\beta \not \in \{ 1,2,\dots ,N, N+1\}$.
Set $c_{-1}(E)=0 $,  $c_0(E)=1 $ and we determine the polynomials $c_n(E)$ $(n=1,\dots ,N+1)$ recursively by Eq.(\ref{eq:rec01}) 
If $E=E_0$ is a solution of the algebraic equation
\begin{equation}
c_{N+1} (E)=0 ,
\label{eq:cN+1}
\end{equation}
then $q$-Heun equation defined in Eq.(\ref{eq:RuijD5}) has a non-zero solution of the form
\begin{equation}
f(x)= x^{\lambda _1} \sum _{n=0}^{N } c_n (E_0)  x^n .
\label{eq:gxxlapol}
\end{equation}
\end{prop}
We call $c_{N+1} (E)$ the spectral polynomial of the $q$-Heun equation.

In general it would be impossible to solve the roots of the spectral polynomial $c_{N+1} (E)$ explicitly.
To understand the roots of the spectral polynomial, we may apply the ultradiscrete limit $q \to +0$.
In \cite{KST}, the behaviour of the roots $E=E_1, E_2, \dots ,E_{N+1} $ of $c_{N+1} (E)=0 $ by the ultradiscrete limit was studied in some cases.
As $q \to 0$, the roots satisfy 
\begin{equation}
E_k \sim -cq^{d-k} ,  \quad k=1,2,\dots ,N+1
\label{eq:Ekintro}
\end{equation}
for some $c \in \R _{>0}$ and $d \in \R $ in those cases.

In this paper we investigate the roots of the spectral polynomial as $q\to +0$ in more cases, which was partially done in \cite{Koj,Sat}.
Namely we obtain results in the three cases in sections \ref{sec:sub1}, \ref{sec:sub2} and \ref{sec:sub3}.
Note that the roots of the spectral polynomial do not satisfy the asymptotics as Eq.(\ref{eq:Ekintro}) in several cases.

This paper is organized as follows.
In section \ref{sec:Equiv}, we introduce two kinds of equivalences on the limit $q\to +0$, which are used to analyze the coefficients of the polynomials $c_j (E)$ $(j=1,2,\dots ,N+1)$.
In section \ref{sec:UDL}, we investigate the polynomials $c_j (E)$ $(j=1,2,\dots ,N+1)$ and the roots of the spectral polynomial $c_{N+1}(E)$ as $q\to +0 $ by dividing into three cases.
In section \ref{sec:rem}, we give concluding remarks. 

Throughout this paper, we assume $0<q<1$.
\section{Equivalences on the ultradiscrete limit} \label{sec:Equiv}

As discussed in \cite{KST}, we define the equivalence $\sim $ of functions of the variable $q$ by 
\begin{equation}
a(q) \sim b (q) \; \Leftrightarrow \; \lim _{q \to +0} \frac{a(q)}{b(q)} =1.
\end{equation}
We also define the equivalence $\sum _{j=0}^{M} a_j(q) E^j \sim \sum _{j=0}^{M} b_j (q) E^j $ of the polynomials of the variable $E$ by $ a_j (q) \sim b _j(q) $ for $j=0,\dots ,M$.
If $\sum _{j=0}^{M} a_j(q) E^j \sim \sum _{j=0}^{M} c_j q^{\mu _j} E^j $ for some $c_j \in \R \setminus \{ 0\}$ $(j=0,1,\dots ,M)$, then we call $\sum _{j=0}^{M} c_j q^{\mu _j} E^j $ the leading terms of $\sum _{j=0}^{M} a_j(q) E^j $.
We introduce another equivalence by 
\begin{equation}
a(q) \approx b (q) \; \Leftrightarrow \; \exists C >0 \mbox{ such that } \lim _{q \to +0} \frac{a(q)}{b(q)} =C ,
\end{equation}
and the equivalence $\sum _{j=0}^{M} a_j(q) E^j \approx \sum _{j=0}^{M} b_j (q) E^j $ by $ a_j (q) \approx b _j(q) $ for $j=0,\dots ,M$.

We are going to find simpler forms of the polynomials $c_{n} (E)$ $(n=1,2,\dots ,N+1)$ determined by Eq.(\ref{eq:rec01}) with respect to the equivalence $\sim $ or $\approx $. 
For simplicity, we assume 
\begin{equation}
N= -\lambda _1 -\alpha _1 \in \Zint _{\geq 0}, \; \beta <1, \; \alpha _2-\alpha _1<1, \; t_1>0 , \; t_2>0, \; h_1 <h_2 , \; l_1<l_2,
\label{eq:assump}
\end{equation} 
throughout this paper, which was also assumed in \cite{KST}.
It follows from Eq.(\ref{eq:rec01}) that the polynomials $c_{n} (E)$ satisfy
\begin{align}
c_n (E) \sim & \: t_1 ^{-1} t_2 ^{-1} [  E q^{n -1 -h_1 -h_2 +\lambda _1} + t_1 q^{1/2 -h_2 } + t_1 q^{ 2n -1/2 -l_2 - \beta  }] c_{n-1} (E) \label{eq:rec00}   \\
&  - t_1^{-1} t_2^{-1} q^{2n-1 -l_1-l_2 -\beta } c_{n-2} (E) \nonumber
\end{align}
for $ n=1,2,\dots ,N+1$ under the assumption that there are no cancellation of the leading terms of the coefficients of $E^j$ $(j=0,1,\dots ,n-1)$ in the right hand side with respect to the limit $q \to +0$.
In \cite{KST}, the leading terms of $c_{n} (E)$ $(n=1,2,\dots ,N+1)$ were investigated for the case ($1 +h_2 -l_2 - \beta >0 $ and $2+ 2 h_2  - l_1 - l_2 -\beta >0  $) and the case ($2N+1 +h_2 -l_2 -\beta <0  $ and $2N + l_1 - l_2 -\beta <0 $).

\section{Analysis of the spectral polynomial by the ultradiscrete limit} \label{sec:UDL}

We investigate the leading terms of $c_{n} (E)$ $(n=1,2,\dots ,N+1)$ for three cases with weaker conditions in the sections \ref{sec:sub1}, \ref{sec:sub2} and \ref{sec:sub3}.

\subsection{The case $1 +h_2 -l_2 - \beta >0 $} \label{sec:sub1} $ $ 

If $1 +h_2 -l_2 - \beta >0 $ and the condition in Eq.(\ref{eq:assump}) is satisfied, then it follows from Eq.(\ref{eq:rec00}) that
\begin{align}
 c_n (E) \sim & \: t_1 ^{-1} t_2 ^{-1} ( E q^{n -1 -h_1 -h_2 +\lambda _1} + t_1 q^{1/2 -h_2 } ) c_{n-1} (E) \label{eq:rec0101} \\
&  - t_1^{-1} t_2^{-1} q^{2n-1 -l_1-l_2 -\beta } c_{n-2} (E) \nonumber 
\end{align}
for $ n=1,2,\dots ,N+1$ under the assumption that there are no cancellation of the leading terms of the coefficients of $E^j$ $(j=0,1,\dots ,n-1)$ in the right hand side with respect to the limit $q \to +0$.
We investigate a sufficient condition that there are no cancellation of the leading terms of the coefficients of $E^j$ $(j=0,1,\dots ,n-1)$.

Since $c_0 (E) =1$ and $c_{-1} (E)=0$, the leading terms of $c_1 (E)$ are described by
\begin{align}
& c_1 (E) \sim t_1 ^{-1} t_2 ^{-1} ( E q^{\lambda _1 -h_1-h_2} + t_1 q^{-h_2 +1/2} ) ,
\end{align}
and we do not need any conditions for no cancellation of the leading terms.
By applying Eq.(\ref{eq:rec0101}), we have 
\begin{align}
c_2 (E) \sim & \: t_1 ^{-2} t_2 ^{-2} ( E q^{1 -h_1 -h_2 +\lambda _1} + t_1 q^{1/2 -h_2 } )( E q^{-h_1 -h_2 +\lambda _1} + t_1 q^{1/2 -h_2 } ) \label{eq:c2Ev1} \\
& - t_1^{-1} t_2^{-1} q^{3 -l_1-l_2 -\beta } \nonumber
\end{align}
under the assumption that there are no cancellation of the leading terms.
In this case, the cancellation of the leading terms may occur on the coefficient of $E^0$ and the candidate for the cancellation is $t_2 ^{-2} q^{1 -2 h_2 } - t_1^{-1} t_2^{-1} q^{3 -l_1-l_2 -\beta } $.  
If $1-2h_2 \neq 3-l_1 -l_2 -\beta $, i.e. $ 2 h_2  - l_1 - l_2 -\beta \neq -2 $, then there are no cancellation of the leading terms.
If $2+ 2 h_2  - l_1 - l_2 -\beta >0 $, then we may ignore the term $t_1^{-1} t_2^{-1} q^{3 -l_1-l_2 -\beta } $ and we have
\begin{align}
& c_2 (E) \sim t_1 ^{-2} t_2 ^{-2} ( E q^{1 -h_1 -h_2 +\lambda _1} + t_1 q^{1/2 -h_2 }) ( E q^{-h_1 -h_2 +\lambda _1} + t_1 q^{1/2 -h_2 } ) .
\end{align}
The leading terms of the polynomials $c_n (E)$ were studied in \cite{KST} for the case $1 +h_2 -l_2 - \beta >0 $ and $2+ 2 h_2  - l_1 - l_2 -\beta >0  $.
If  $2+ 2 h_2  - l_1 - l_2 -\beta <0  $, then we have
\begin{align}
& c_2 (E) \sim t_1 ^{-2} t_2 ^{-2} [ q^{1 -2h_1 -2h_2 +2 \lambda _1} E^2 + t_1 q^{1/2 -h_2 } q^{-h_1 -h_2 +\lambda _1} E  - t_1 t_2 q^{3 -l_1-l_2 -\beta } ] .
\end{align}
We assume $2 h_2  - l_1 - l_2 -\beta \neq -2 $. 
It follows from Eq.(\ref{eq:rec0101}) for $n=3$ that
\begin{align}
& c_3 (E) \sim c_{2} (E) t_1^{-1} t_2^{-1} ( E q^{2 +\lambda _1 - h_1 - h_2} +  t_1 q^{-h_2 +1/2} ) - c_{1} (E)  t_1^{-1} t_2^{-1} q^{5 -l_1 - l_2 -\beta }  \\
& \sim t_1^{-3} t_2^{-3} ( E q^{2 +\lambda _1 - h_1 - h_2} +  t_1 q^{-h_2 +1/2} ) ( E q^{1 +\lambda _1  -h_1-h_2} +  t_1 q^{-h_2 +1/2} ) \nonumber \\
& \qquad \qquad ( E q^{\lambda _1 -h_1-h_2} +  t_1 q^{-h_2 +1/2} ) \nonumber \\
& \quad -  t_1^{-2} t_2^{-2} q^{5 -l_1 - l_2 -\beta  } ( E q^{\lambda _1 -h_1-h_2} +  t_1 q^{-h_2 +1/2} ) \nonumber \\
& \quad -  t_1^{-2} t_2^{-2} q^{3 -l_1 - l_2 -\beta  } ( E q^{2 +\lambda _1 - h_1 - h_2} +  t_1 q^{-h_2 +1/2} ) \nonumber \\
& \sim  t_1^{-3} t_2^{-3}[ E^3 q^{3 +3(\lambda _1 - h_1 - h_2)} + E^2 t_1 q^{1 +2 (\lambda _1  -h_1-h_2)} q^{-h_2 +1/2} \nonumber \\
& \quad  +E  t_1^2 q^{\lambda _1  -h_1-h_2} q^{2(-h_2 +1/2)}+ t_1 ^{3} q^{3(-h_2 +1/2)} \nonumber \\
& \quad  - t_1 t_2 q^{-l_1 - l_2 -\beta } ( 2 E q^{5 +\lambda _1 - h_1 - h_2} + t_1 q^{3-h_2 +1/2} ) ] \nonumber 
\end{align}
under the assumption that there are no cancellations of the leading terms.
In this case, the cancellation of the leading terms may occur on the coefficients of $E^1$ and $E^0$.
Hence if $2 h_2  - l_1 - l_2 -\beta \not \in \{ -4,-2 \}$, then the cancellation of the leading terms does not occur.
On the polynomial $c_4(E)$, we have
\begin{align}
& c_4 (E) \sim c_{3} (E) t_1^{-1} t_2^{-1} ( E q^{3 +\lambda _1 - h_1 - h_2} +  t_1 q^{-h_2 +1/2} ) -c_{2} (E)  t_1^{-1} t_2^{-1} q^{7 -l_1 - l_2 -\beta  } \label{eq:c4-1} \\
& \sim t_1^{-4} t_2^{-4} ( E q^{\lambda _1 -h_1-h_2} +  t_1 q^{-h_2 +1/2} ) ( E q^{1 +\lambda _1  -h_1-h_2} +  t_1 q^{-h_2 +1/2} ) \nonumber \\
& \qquad \qquad ( E q^{2 +\lambda _1 - h_1 - h_2} + t_1 q^{-h_2 +1/2} )  ( E q^{3 +\lambda _1 - h_1 - h_2} +  t_1 q^{-h_2 +1/2} ) \nonumber \\
& \quad - t_1^{-3} t_2^{-3} q^{3 -l_1 - l_2 -\beta } (  E q^{2 +\lambda _1 - h_1 - h_2} + t_1 q^{-h_2 +1/2} ) ( E q^{3 +\lambda _1 - h_1 - h_2} +  t_1 q^{-h_2 +1/2} ) \nonumber \\
& \quad - t_1^{-3} t_2^{-3} q^{5 -l_1 - l_2 -\beta } (  E q^{\lambda _1 - h_1 - h_2} + t_1 q^{-h_2 +1/2} ) ( E q^{3 +\lambda _1 - h_1 - h_2} +  t_1 q^{-h_2 +1/2} ) \nonumber \\
& \quad - t_1^{-3} t_2^{-3} q^{7 -l_1 - l_2 -\beta  } ( E q^{\lambda _1  -h_1-h_2} +  t_1 q^{-h_2 +1/2} ) ( E q^{1+ \lambda _1 -h_1-h_2} +  t_1 q^{-h_2 +1/2} ) \nonumber \\
& \quad + t_1^{-2} t_2^{-2} q^{7 -l_1 - l_2 -\beta  } q^{3 -l_1 - l_2 -\beta  } \nonumber \\
& \sim t_1^{-4} t_2^{-4} [ E^4 q^{6 +4(\lambda _1 - h_1 - h_2)} + E^3 t_1 q^{3(\lambda _1 - h_1 - h_2)} q^{3-h_2 +1/2} \nonumber  \\ 
& \quad + E^2 t_1^{2} q^{2(\lambda _1 - h_1 - h_2)} q^{1+ 2(-h_2 +1/2)}  + E t_1^{3} q^{\lambda _1 - h_1 - h_2} q^{3(-h_2 +1/2)} \nonumber \\
& \quad + t_1 ^{4} q^{4(-h_2 +1/2)} - t_1 t_2 q^{-l_1 - l_2 -\beta } \{ 3 E^2  q^{8 + 2(\lambda _1 - h_1 - h_2)}  \nonumber \\
& \qquad \qquad +  2 E  t_1 q^{5-h_2 +1/2} q^{\lambda _1 - h_1 - h_2} + t_1 ^{2} q^{3+2(-h_2 +1/2)} \} + t_1^{2} t_2^{2} q^{10 -2(l_1 + l_2 +\beta ) } ] \nonumber 
\end{align}
under the assumption that there are no cancellations of the leading terms.
The cancellation of the leading terms may occur on the coefficients of $E^2$, $E^1$ and $E^0$, and the cancellation of the leading terms does not occur under the condition $2 h_2  - l_1 - l_2 -\beta \not \in \{ -6, -4,-2 \}$.
On the polynomial $c_n (E)$, we have the following proposition.
\begin{prop} \label{prop:noncancel01}
Let $M \in \{ 1,2, \dots , N+1 \}$.
A sufficient condition that the cancellation of the leading terms in the right hand side of Eq.(\ref{eq:rec0101}) does not occur for $n=1,2,\dots ,M$ is written as $2 h_2  - l_1 - l_2 -\beta \not \in \{ -2M+2, -2M+4, \dots ,-4,-2 \}$.
\end{prop}
\begin{proof}
We can write the recursive relation as 
\begin{equation}
c_n(E) = (p_n E +q_n) c_{n-1} (E) -r_n c_{n-2} (E),
\label{eq:cnpnqnrn}
\end{equation}
where
\begin{equation}
r_n \sim t_1^{-1} t_2^{-1} q^{2n-1 -l_1 - l_2 -\beta }, \; p_n \sim t_1^{-1} t_2^{-1} q^{n-1 +\lambda _1 - h_1 - h_2} , \; q_n \sim  t_2 ^{-1} q^{-h_2 +1/2} .
\end{equation}
By applying Eq.(\ref{eq:cnpnqnrn}) repeatedly, we have
\begin{equation}
c_M (E) \sim \sum _{0\leq k \leq M/2} (-1)^k \sum \!  ' \prod _{l=1}^k r_{n_l +1} \prod _m \! '' (p_m E+q_m) ,
\label{eq:cMsum}
\end{equation}
where the summation $\sum ' $ is over the integers $1\leq n_1<n_2 < \dots <n_k \leq M-1 $ such that $n_l - n_{l-1} \geq 2 $ for $l=2,\dots k$ and the product $\prod \! '' _m  $ is over the integer $1 \leq m \leq M$ such that $m \not \in \{ n_l , n_l +1 \}$ for $l=1,\dots k$.
The term $r_n (p_{n+1} E+q _{n+1}) \sim t_1^{-2} t_2^{-2} q^{2n-1 -l_1 - l_2 -\beta } (q^{n +\lambda _1 - h_1 - h_2} E+ t_1 q^{-h_2 +1/2}) $ is stronger than $(p_{n-1} E+q _{n-1}) r_{n+1} \sim t_1^{-2} t_2^{-2} q^{2n-1 -l_1 - l_2 -\beta } (q^{n +\lambda _1 - h_1 - h_2} E+ t_1 q^{-h_2 +5/2}) $.
By applying this relation repeatedly, we find that the strongest term in Eq.(\ref{eq:cMsum}) for the fixed $k$ is contained in
\begin{align}
& r_2 r_4 \dots r_{2k} ( p_{2k +1} E +q_{2k+1}) ( p_{2k +2} E +q_{2k+2}) \dots ( p_{M} E +q_{M}) \\
& \approx  \sum _{0 \leq l \leq M-2k } r_2 r_4 \dots r_{2k}  p_{2k +1} p_{2k+2} \dots p_{2k+l} q_{2k+l+1} \dots q_{M} E^l \nonumber \\
& \approx  \sum _{0 \leq l \leq M-2k } q^{d(k,l)} E^l, \nonumber 
\end{align}
where $d(k,l) = k(2k+ 1- l_1 - l_2 -\beta ) + l(l-1)/2 + l (2k -h_1-h_2+\lambda _1 ) + (M-2k-l) (-h_2 +1/2 )$.
Hence we have
\begin{equation}
c_M (E) \approx \sum _{l ,k  \geq 0,  0 \leq l +2k \leq M } (-1)^k  q^{d(k,l)} E^l  .
\end{equation}
If the cancellation of the leading terms occur, then we have $d(k,l) = d(k',l)$ for $k \neq k'$, i.e.,
$2(k+k') + 2l + 2 h_2 -  (l_1 + l_2 +\beta ) = 0$  for $k \neq k'$, $0 \leq 2k +l \leq M$ and $0 \leq 2k' +l \leq M$.
Since $k\neq k'$, we have $0< k+k' +l <M$ and a sufficient condition that the cancellation of the leading terms does not occur is written as $2 h_2  - l_1 - l_2 -\beta \not \in \{ -2M+2, -2M+4, \dots ,-4,-2 \}$.
\end{proof}
We consider the expression of the leading terms roughly by using the equivalence $\approx $ under the condition of Proposition \ref{prop:noncancel01}.
On the polynomials $c_2(E)$ and $c_3(E)$, we have 
\begin{align}
  c_2 (E) \sim & \: t_1^{-2} t_2^{-2} ( E q^{1 +\lambda _1  -h_1-h_2} +  t_1 q^{-h_2 +1/2} )( E q^{\lambda _1 -h_1-h_2} +  t_1 q^{-h_2 +1/2} )  \label{eq:c2gen} \\
&  - t_1^{-1} t_2^{-1} q^{3 -l_1 - l_2 -\beta  } \nonumber \\
\approx & \: ( E q^{1 +\lambda _1  -h_1-h_2} + q^{-h_2 +1/2} )( E q^{\lambda _1 -h_1-h_2} + q^{-h_2 +1/2} ) - q^{3 -l_1 - l_2 -\beta  },
 \nonumber \\
 c_3 (E) \sim & \: t_1^{-3} t_2^{-3} ( E q^{2 +\lambda _1 - h_1 - h_2} + t_1 q^{-h_2 +1/2} ) ( E q^{1 +\lambda _1  -h_1-h_2} +  t_1 q^{-h_2 +1/2} ) \nonumber \\ 
& \qquad \qquad ( E q^{\lambda _1 -h_1-h_2} +  t_1 q^{-h_2 +1/2} )  \nonumber \\ 
& - t_1^{-2} t_2^{-2} q^{3 -l_1 - l_2 -\beta } ( 2 E q^{2 +\lambda _1 - h_1 - h_2} + t_1 q^{-h_2 +1/2} ) \nonumber  \\
\approx & \: ( E q^{2 +\lambda _1 - h_1 - h_2} + q^{-h_2 +1/2} ) ( E q^{1 +\lambda _1  -h_1-h_2} + q^{-h_2 +1/2} ) ( E q^{\lambda _1 -h_1-h_2} + q^{-h_2 +1/2} ) \nonumber \\ 
& - q^{3 -l_1 - l_2 -\beta } ( E q^{2 +\lambda _1 - h_1 - h_2} + q^{-h_2 +1/2} ) \nonumber \\
\approx & \: ( E q^{2 +\lambda _1 - h_1 - h_2} +  q^{-h_2 +1/2} ) c_2 (E) . \nonumber 
%
\end{align}
We also have
\begin{align}
 c_4 (E) \sim & \: c_{3} (E) t_1^{-1} t_2^{-1} ( E q^{3 +\lambda _1 - h_1 - h_2} +  t_1 q^{-h_2 +1/2} ) -c_{2} (E) t_1^{-1} t_2^{-1} q^{7 -l_1 - l_2 -\beta  }\\
 \approx & \: c_2 (E) [ ( E q^{3 +\lambda _1 - h_1 - h_2} + q^{-h_2 +1/2} ) ( E q^{2 +\lambda _1 - h_1 - h_2} +  q^{-h_2 +1/2} ) - q^{7 -l_1 - l_2 -\beta } ].  \nonumber 
%
\end{align}
These relations are generalized as follows.
\begin{prop} \label{prop:reccnappr01}
Assume that $1 +h_2 -l_2 - \beta >0 $ and $2 h_2  - l_1 - l_2 -\beta \not \in \{ -2N, -2N+2 ,\dots ,-4,-2 \}$.\\
(i) If $n \in \Zint _{\geq 1}$ and $2n <N+1 $, then
\begin{align}
 c_{2n} (E) \approx & \: c_{2n-2} (E) [ ( E q^{2n-1 +\lambda _1 - h_1 - h_2} +q^{-h_2 +1/2} ) \\
& \qquad \qquad ( E q^{2n-2 +\lambda _1 - h_1 - h_2} +q^{-h_2 +1/2} ) - q^{4n-1 -l_1 - l_2 -\beta } ] . \nonumber 
\end{align}
(ii) If $n \in \Zint _{\geq 1}$ and $2n <N $, then
\begin{align}
& c_{2n+1} (E) \approx c_{2n} (E) ( E q^{2n +\lambda _1 - h_1 - h_2} +q^{-h_2 +1/2} ) .
\end{align}
\end{prop}
\begin{proof}
The case $n=1$ in (i) and (ii) was shown by Eq.(\ref{eq:c2gen}).
We show the formulas for $n=m+1$ by assuming those for $n=m$.
It follows from $2 h_2  - l_1 - l_2 -\beta \not \in \{ -2N, -2N+2 ,\dots ,-4,-2 \}$ that there are no cancellations of the leading terms of the cofficients $E^j$ ($j=0, 1, \dots  $) on the right hand side of Eq.(\ref{eq:rec0101}) for $n=2m +2$.
Hence it follows from Eq.(\ref{eq:rec0101}) and $c_{2m+1} (E) \approx c_{2m} (E) ( E q^{2m +\lambda _1 - h_1 - h_2} +q^{-h_2 +1/2} )  $ that
\begin{align} 
 c_{2m+2} (E) \approx & \: c_{2m+1} (E) ( E q^{2m +1 +\lambda _1 - h_1 - h_2} +q^{-h_2 +1/2} ) - c_{2m} (E) q^{4m +3 -l_1 - l_2 -\beta } \nonumber \\
 \approx & \: c_{2m} (E) [( E q^{2m+1 +\lambda _1 - h_1 - h_2} +q^{-h_2 +1/2} ) \\
& \qquad \qquad (E q^{2m +\lambda _1 - h_1 - h_2} +q^{-h_2 +1/2} )- q^{4m +3 -l_1 - l_2 -\beta } ] . \nonumber 
\end{align}
Therefore we obtain (i) for $n=m+1$.

It follows from Eq.(\ref{eq:rec0101}) for $n=2m+3, 2m+2$ and $c_{2m+1} (E) \approx c_{2m} (E) ( E q^{2m +\lambda _1 - h_1 - h_2} +q^{-h_2 +1/2} )  $ that
\begin{align} 
& c_{2m+3} (E) \approx  c_{2m+2} (E) ( E q^{2m+2 +\lambda _1 - h_1 - h_2} +q^{-h_2 +1/2} )  - c_{2m+1} (E) q^{4m +5 -l_1 - l_2 -\beta } \\
& \approx  c_{2m+1} (E) ( E q^{2m +1 +\lambda _1 - h_1 - h_2} +q^{-h_2 +1/2} )( E q^{2m+2 +\lambda _1 - h_1 - h_2} +q^{-h_2 +1/2} )  \nonumber \\
& \; - c_{2m} (E) q^{4m +3 -l_1 - l_2 -\beta } ( E q^{2m+2 +\lambda _1 - h_1 - h_2} +q^{-h_2 +1/2} ) \nonumber \\
& \; -c_{2m} (E) q^{4m +5 -l_1 - l_2 -\beta }  ( E q^{2m +\lambda _1 - h_1 - h_2} +q^{-h_2 +1/2} ) . \nonumber
%
\end{align}
Since the term $q^{4m +3 -l_1 - l_2 -\beta } ( E q^{2m+2 +\lambda _1 - h_1 - h_2} +q^{-h_2 +1/2} ) $ is stronger than the term $q^{4m +5 -l_1 - l_2 -\beta }  ( E q^{2m +\lambda _1 - h_1 - h_2} +q^{-h_2 +1/2} )  $, we may neglect the term $ c_{2m} (E) q^{4m +5 -l_1 - l_2 -\beta }$ $ ( E q^{2m +\lambda _1 - h_1 - h_2} + q^{-h_2 +1/2} ) $ under the equivalence $\approx $ and we have $c_{2m+3} (E) \approx c_{2m+2} (E) $ $( E q^{2m+2 +\lambda _1 - h_1 - h_2} +q^{-h_2 +1/2} ) $.
\end{proof}
By applying Proposition \ref{prop:reccnappr01} repeatedly, we obtain the approximation of the spectral polynomial $c_{N+1} (E) $ as follows.
\begin{thm}
We assume Eq.(\ref{eq:assump}), $1 +h_2 -l_2 - \beta >0 $ and $2 h_2  - l_1 - l_2 -\beta \not \in \{ -2N, -2N+2 ,\dots ,-4,-2 \}$.
Set 
\begin{align}
p _n (E)& = ( E q^{2n-1 +\lambda _1 - h_1 - h_2} +q^{-h_2 +1/2} ) ( E q^{2n-2 +\lambda _1 - h_1 - h_2} +q^{-h_2 +1/2} ) \nonumber \\
& \qquad \qquad \qquad - q^{4n-1 -l_1 - l_2 -\beta } , \label{eq:pntildecN+1} \\
\tilde{c}_{N+1} (E) & = \left\{ \begin{array}{ll}
\displaystyle \prod_{n=1}^{(N+1)/2} p _n (E), & \mbox{$N$ is odd,} \\
\displaystyle ( E q^{N +\lambda _1 - h_1 - h_2} +q^{-h_2 +1/2} ) \prod_{n=1}^{N/2} p _n (E) , & \mbox{$N$ is even.} 
\end{array}
\right. \nonumber 
\end{align}
Then we have $c_{N+1} (E) \approx \tilde{c}_{N+1} (E)$.

\end{thm}
We investigate the zeros of $p_n (E)$ as $q\to +0$.
If $ 2h_2 -l_1 - l_2 -\beta >2 -4n$ (i.e., $-2h_2 +1 < 4n-1 -l_1 - l_2 -\beta $), then 
\begin{equation}
p _n (E) \sim  q^{4n-3 +2(\lambda _1 - h_1 - h_2)} ( E +q^{-2n+3/2 -\lambda _1 + h_1 } ) ( E +q^{-2n+5/2 -\lambda _1 + h_1} ) .
\end{equation}
If $ 2h_2 -l_1 - l_2 -\beta <2 -4n$, then we have $p _n (E) \sim q^{4n-3 +2(\lambda _1 - h_1 - h_2)} (E^2 + E q^{-2n +3/2 -\lambda _1 + h_1 }  - q^{2 -l_1 - l_2 -\beta  -2\lambda _1 +2 h_1 +2 h_2} ) $.
Moreover, if $1< 4 n + 2h_2 -l_1 - l_2 -\beta  <2 $, then 
\begin{align}
p _n (E) \sim & q^{4n-3 +2(\lambda _1 - h_1 - h_2)} (E+ q^{-2n +3/2 -\lambda _1+ h_1 } ) (E- q^{2n +1/2 -\lambda _1 -l_1 - l_2 -\beta  + h_1 + 2 h_2  }).
\label{eq:14n2}
%
\end{align}
In the case $4 n + 2h_2 -l_1 - l_2 -\beta  <1 $, the polynomial $p _n (E)$ is not factorized as Eq.(\ref{eq:14n2}) as $q\to +0 $, and we solve the quadratic equation $E^2 + E q^{-2n +3/2 -\lambda _1 + h_1 }  - q^{2 -l_1 - l_2 -\beta  -2\lambda _1 +2 h_1 +2 h_2}=0 $ as $q\to 0$ by the quadratic formula.
It follows from  $ 4 n + 2h_2 -l_1 - l_2 -\beta  <1 $ that
\begin{align}
E & = \frac{q^{-2n +3/2 -\lambda _1 + h_1 } \pm \sqrt{ q^{2(-2n +3/2 -\lambda _1 + h_1 )}+ 4 q^{2 -l_1 - l_2 -\beta  -2\lambda _1 +2 h_1 +2 h_2} }}{2}  \nonumber \\
& \sim \frac{q^{-2n +3/2 -\lambda _1 + h_1 } \pm \sqrt{ 4 q^{2 -l_1 - l_2 -\beta  -2\lambda _1 +2 h_1 +2 h_2} }}{2} \\
& \sim \pm  q^{1 -l_1/2 - l_2/2 -\beta /2 -\lambda _1 + h_1 + h_2} = \pm q^{(\alpha _1 + \alpha _2  +h_1 + h_2 )/2} , \nonumber
\end{align} 
where we used $\lambda _1 = (h_1 +h_2 -l_1-l_2 -\alpha _1-\alpha _2 -\beta +2)/2$.

We describe the zeros of polynomial $\tilde{c}_{N+1} (E) $ defined in Eq.(\ref{eq:pntildecN+1}) as $q\to +0$ by dividing into parts.\\
(i) If $ 2h_{2} - l_{1} - l_{2}  -\beta > -2 $, then 
\begin{align}
& E \sim  q^{-j+3/2-\lambda _1  +h_{1}} \ \ \ \left( 1 \leq j \leq N + 1 \right).
\label{eq:asym1-1}
\end{align}
(ii-1) If $ 2h_{2} - l_{1} - l_{2}  -\beta < - 2N -1  $ and $N$ is odd, then
\begin{equation}
  E \sim - q^{( \alpha _1 +\alpha _2 +h_{1}+h_{2})/2}, \ q^{(\alpha _1 +\alpha _2 +h_{1}+h_{2})/2} \ \ \ \left( \mbox{$(N + 1)/2$-ple} \right).
\label{eq:asym1-2-1}
\end{equation}
(ii-2) If $ 2h_{2} - l_{1} - l_{2} -\beta < - 2N + 1 $, $ 2h_{2} - l_{1} - l_{2} -\beta \neq - 2N  $ and $N$ is even, then
\begin{align}
  E  \sim & - q^{(\alpha _1 +\alpha _2 +h_{1}+h_{2})/2}, \ q^{(\alpha _1 +\alpha _2 +h_{1}+h_{2})/2} \ \ \ \left( \mbox{$N/2$-ple} \right),\\
          &  - q^{-N+1/2-\lambda _1 +h_{1}}. \nonumber 
\end{align}
(iii-1) If $ - 4m+1 < 2h_{2} - l_{1} - l_{2} -\beta < - 4m + 2 $ for some $m \in \Zint $ such that $1\leq m \leq (N+1)/2$, then 
\begin{align}
  E &\sim - q^{(\alpha _1 +\alpha _2 +h_{1}+h_{2})/2}, \ q^{(\alpha _1 +\alpha _2 +h_{1}+h_{2})/2}, \ \ \ \left( \mbox{$(m-1)$-ple} \right) \\
    & \ \ \ \ q^{2m-3/2+\lambda _1+ \alpha _1 +\alpha _2 +h_{2}}, \nonumber \\
    & \ \ \ \ - q^{-j+3/2-\lambda _1 +h_{1}} \ \ \ \left( 2m \leq j \leq N + 1 \right). \nonumber 
\end{align}
(iii-2) If $ - 4m -2 <  2h_{2} - l_{1} - l_{2} -\beta < - 4m + 1 $ for some $m \in \Zint $ such that $1\leq m \leq (N-1)/2$ and $2h_{2} - l_{1} - l_{2} -\beta \neq - 4m $, then 
\begin{align}
  E &\sim - q^{(\alpha _1 +\alpha _2 +h_{1}+h_{2})/2}, \ q^{(\alpha _1 +\alpha _2 +h_{1}+h_{2})/2}, \ \ \ \left( \mbox{$m$-ple} \right) \\
    & \ \ \ \ - q^{-j+3/2-\lambda _1 +h_{1}} \ \ \ \left( 2m + 1 \leq j \leq N + 1 \right). \nonumber 
\end{align} 

We discuss the zeros of the polynomials $c_{N+1} (E) $ as $q\to +0$ by comparing with the zeros of $ \tilde{c}_{N+1} (E) $ described above.
If $ 2h_{2} - l_{1} - l_{2}  -\beta > -2 $, then it was shown in \cite{KST} that the zeros of the polynomial $c_{N+1} (E) $ are written as $E_j(q)$ $(j=1,\dots N+1)$ such that $E_j(q) \sim -t_1 q^{-j+3/2-\lambda _1 +h_{1}}$ for sufficiently small $q(>0)$.
Hence we have $E_j(q) \approx - q^{-j+3/2-\lambda _1 +h_{1}}$ for $j=1, \dots , N+1$ and it is compatible with Eq.(\ref{eq:asym1-1}).

Although the multiplicity of the roots of $\tilde{c}_{N+1} (E) $ for $q=0$ appears in the case $2h_{2} - l_{1} - l_{2}  -\beta > -7 $, the roots of $\tilde{c}_{N+1} (E) $ for $q>0$ do not have multiplicity, which follows from the real root property of the spectral polynomial discussed in \cite{KST}.
We give an example.
In the case $N=3$ and $2h_{2} - l_{1} - l_{2}  -\beta > -7 $, Eq.(\ref{eq:c4-1}) is written as 
\begin{align}
& c_4 (E) \sim t_1^{-4} t_2^{-4} [E^4  q^{6 +4(\lambda _1 - h_1 - h_2)} + E^3 t_1 q^{3(\lambda _1 - h_1 - h_2)} q^{3-h_2 +1/2} \label{eq:c4-1-1} \\ 
& \quad - t_1 t_2 q^{-l_1 - l_2 -\beta } \{ 3 E^2  q^{8 + 2(\lambda _1 - h_1 - h_2)} +  2 E  t_1 q^{5-h_2 +1/2} q^{\lambda _1 - h_1 - h_2} \} \nonumber \\
& \quad + t_1^{2} t_2^{2} q^{10 -2(l_1 + l_2 +\beta ) }] , \nonumber 
\end{align}
and the zeros of the right hand side of Eq.(\ref{eq:c4-1-1}) satisfy $E \approx  +q^{(\alpha _1 +\alpha _2 +h_{1}+h_{2})/2},$ $  - q^{(\alpha _1 +\alpha _2 +h_{1}+h_{2})/2} $ with multiplicity two (see Eq.(\ref{eq:asym1-2-1})).
To obtain more detailed asymptotics, we set $E= s q^{(\alpha _1 +\alpha _2 +h_{1}+h_{2})/2} $, substitute it into the right hand side and observe the condition that the leading term disappear. Then we have
\begin{align}
& s^4 -3  t_1 t_2 s^2 + t_1^{2} t_2^{2} =0.
\end{align}
Hence $E= \pm q^{(\alpha _1 +\alpha _2 +h_{1}+h_{2})/2}  (t_1 t_2)^{1/2} (\sqrt{5} + 1) /2 $ and $E= \pm q^{(\alpha _1 +\alpha _2 +h_{1}+h_{2})/2} $ $ (t_1 t_2)^{1/2} (\sqrt{5} - 1) /2  $ are the roots of the right hand side of Eq.(\ref{eq:c4-1-1}), which do not have multiplicity.

\subsection{The case $2N+1 +h_2 -l_2 -\beta <0 $} \label{sec:sub2} $ $

If $2N+1 +h_2 -l_2 -\beta <0 $ and the condition in Eq.(\ref{eq:assump}) is satisfied, then 
\begin{align}
 c_n (E) \sim & t_1^{-1} t_2^{-1} ( E q^{n -1 -h_1 -h_2 +\lambda _1} + q^{2n -1/2 -l_2 - \beta  } t_1 ) c_{n-1} (E)  \label{eq:rec0102} \\
& - q^{2n-1 -l_1-l_2 -\beta } t_1 ^{-1} t_2 ^{-1} c_{n-2} (E) \nonumber
\end{align}
for $n=1,2,\dots ,N+1$ under the assumption of no cancellations.
Then we have
\begin{align}
 c_1 (E) \sim & t_1^{-1} t_2^{-1} (E q^{ -h_1 -h_2 +\lambda _1} + q^{3/2 -l_2 - \beta  } t_1), \\
 c_2 (E) \sim & t_1^{-2} t_2^{-2} ( E q^{ -h_1 -h_2 +\lambda _1} + q^{3/2 -l_2 - \beta  } t_1 ) ( E q^{1 -h_1 -h_2 +\lambda _1} + q^{7/2 -l_2 - \beta  } t_1 ) \nonumber \\
& - q^{3 -l_1-l_2 -\beta } t_1 ^{-1} t_2 ^{-1} , \nonumber 
\end{align}
and the candidate of the cancellation is the term $t_2 ^{-2} q^{5 -2 l_2 - 2 \beta  } - t_1^{-1} t_2^{-1} q^{3 -l_1-l_2 -\beta } $. 
Hence there are no cancellation of the leading terms for $c_2(E)$ in the case $l_1 - l_2 -\beta \neq -2$.
If  $2+ l_1 - l_2 -\beta <0 $, then we may ignore the term $t_1^{-1} t_2^{-1} q^{3 -l_1-l_2 -\beta } $ and we have
\begin{align}
& c_2 (E) \sim t_1^{-2} t_2^{-2} ( E q^{ -h_1 -h_2 +\lambda _1} + q^{3/2 -l_2 - \beta  } t_1 ) ( E q^{1 -h_1 -h_2 +\lambda _1} + q^{7/2 -l_2 - \beta  } t_1 ).
\end{align}
If  $2+ l_1 - l_2 -\beta >0 $, then
\begin{align}
& c_2 (E) \sim t_1^{-2} t_2^{-2} [ q^{1 -2h_1 -2h_2 +2 \lambda _1} E^2 + t_1 q^{5/2 -h_1 -h_2 +\lambda _1 -l_2 - \beta  } E  - t_1 t_2 q^{3 -l_1-l_2 -\beta } ]  .
\end{align}
We assume $l_1 - l_2 -\beta \neq -2 $ and apply Eq.(\ref{eq:rec0102}) for $n=3$. Then 
\begin{align}
& c_3 (E) \sim c_{2} (E) t_1^{-1} t_2^{-1} ( E q^{2 +\lambda _1 - h_1 - h_2} +  t_1 q^{11/2 -l_2 - \beta  })  - c_{1} (E)  t_1^{-1} t_2^{-1} q^{5 -l_1 - l_2 -\beta }  \nonumber \\
& \sim t_1^{-3} t_2^{-3} ( E q^{ -h_1 -h_2 +\lambda _1} + q^{3/2 -l_2 - \beta  } t_1 ) ( E q^{1 -h_1 -h_2 +\lambda _1} + q^{7/2 -l_2 - \beta  } t_1 )  \nonumber \\
& \qquad \qquad  ( E q^{2 +\lambda _1 - h_1 - h_2} +  t_1 q^{11/2 -l_2 - \beta  } ) \\
& \quad -  t_1^{-2} t_2^{-2} q^{3 -l_1 - l_2 -\beta  } ( E q^{2 - h_1 - h_2 +\lambda _1} +  t_1 q^{11/2 -l_2 - \beta  }) \nonumber \\
& \quad -  t_1^{-2} t_2^{-2} q^{5 -l_1 - l_2 -\beta  } ( E q^{ -h_1 -h_2 +\lambda _1} + t_1 q^{3/2 -l_2 - \beta  }  ) \nonumber \\
& \sim t_1^{-3} t_2^{-3} [ E^3 q^{3 +3(\lambda _1 - h_1 - h_2)} + E^2 t_1 q^{2 (\lambda _1  -h_1-h_2)} q^{9/2-(l_2 + \beta ) } \nonumber \\
& \quad +E t_1^{2} q^{ \lambda _1  -h_1-h_2} q^{7-2(l_2 + \beta )} + t_1 ^{3} q^{21/2-3(l_2 + \beta )}  \nonumber \\
& \quad - t_1 t_2 q^{5 -l_1 - l_2 -\beta } ( 2 E q^{ \lambda _1 - h_1 - h_2} + t_1 q^{3/2 -l_2 - \beta } ) ]  \nonumber 
\end{align}
under the assumption that there are no cancellation of the leading terms.
Then the condition $l_1 - l_2 -\beta  \not \in \{ -4,-2 \} $ is sufficient for non-cancellation of the leading terms, and it is generalized as follows.
\begin{prop} \label{prop:noncancel02}
Let $M \in \{ 1,2, \dots , N+1 \}$.
A sufficient condition that the cancellation of the leading terms in the right hand side of Eq.(\ref{eq:rec0102}) for $n=1,2,\dots ,M$ does not occur is written as $l_1 - l_2 -\beta   \not \in \{ -2M+2, -2M+4, \dots ,-4,-2 \}$.
\end{prop}
\begin{proof}
We can prove the proposition similarly to Proposition \ref{prop:noncancel01}, although we need a slight modification.
The recursive relation for $c_n(E)$ is written as Eq.(\ref{eq:cnpnqnrn}) where
\begin{equation}
r_n \sim t_1^{-1} t_2^{-1} q^{2n-1 -l_1 - l_2 -\beta }, \; p_n \sim t_1^{-1} t_2^{-1} q^{n-1 +\lambda _1 - h_1 - h_2} , \; q_n \sim t_2 ^{-1}  q^{2n -1/2 -l_2 - \beta  } .
\end{equation}
By applying Eq.(\ref{eq:cnpnqnrn}) repeatedly, we have the expression as Eq.(\ref{eq:cMsum}).
In this case, the term $(p_{n-1} E+q _{n-1}) r_{n+1} \sim t_1^{-2} t_2^{-2} q^{2n+1 -l_1 - l_2 -\beta } (q^{n -2 +\lambda _1 - h_1 - h_2} E+ t_1 q^{2n -5/2 -l_2 - \beta }) $ is stronger than $ r_n (p_{n+1} E+q _{n+1}) \sim t_1^{-2} t_2^{-2} q^{2n-1 -l_1 - l_2 -\beta } (q^{n +\lambda _1 - h_1 - h_2} E+ t_1 q^{2n +3/2 -l_2 - \beta }) $.
By applying this relation repeatedly, we find that the strongest term in Eq.(\ref{eq:cMsum}) for the fixed $k$ is contained in
\begin{align}
&  ( p_{1} E +q_{1}) ( p_{2} E +q_{2}) \dots ( p_{M-2k} E +q_{M-2k}) r_{M-2k+2} r_{M-2k+4} \dots r_{M} \\
& \approx  \sum _{0 \leq l \leq M-2k } q_{1} q_{2} \dots q_{M-2k-l} p_{M-2k-l+1} \dots p_{M-2k}  r_{M-2k+2} r_{M-2k+4} \dots r_{M} E^l .  \nonumber
\end{align}
By repeating the argument in the proof of Proposition \ref{prop:noncancel01}, we can obtain Proposition \ref{prop:noncancel02}.
\end{proof}
We consider the expression of the leading terms roughly by using the equivalence $\approx $ under the condition of Proposition \ref{prop:noncancel02}.
On the polynomials $c_2(E)$ and $c_3(E)$, we have 
\begin{align}
& c_2 (E) \sim t_1^{-2} t_2^{-2} ( E q^{ -h_1 -h_2 +\lambda _1} + q^{3/2 -l_2 - \beta  } t_1 ) ( E q^{1 -h_1 -h_2 +\lambda _1} + q^{7/2 -l_2 - \beta  } t_1 )  \label{eq:c2gen2} \\
& \qquad \qquad - q^{3 -l_1-l_2 -\beta } t_1 ^{-1} t_2 ^{-1} \nonumber \\
& \quad \approx ( E q^{ -h_1 -h_2 +\lambda _1} + q^{3/2 -l_2 - \beta  } ) ( E q^{1 -h_1 -h_2 +\lambda _1} + q^{7/2 -l_2 - \beta  } ) - q^{3 -l_1-l_2 -\beta } , \nonumber \\
& c_3 (E) \sim t_1^{-3} t_2^{-3} ( E q^{ -h_1 -h_2 +\lambda _1} + q^{3/2 -l_2 - \beta  } t_1 ) ( E q^{1 -h_1 -h_2 +\lambda _1} + q^{7/2 -l_2 - \beta  } t_1 ) \nonumber \\
& \qquad \qquad \qquad ( E q^{2 +\lambda _1 - h_1 - h_2} + q^{11/2 -l_2 - \beta  } t_1 ) \nonumber \\
& \qquad \qquad - t_1^{-2} t_2^{-2} q^{5 -l_1 - l_2 -\beta } ( 2 E q^{ \lambda _1 - h_1 - h_2} + q^{3/2 -l_2 - \beta } t_1)  \nonumber \\
& \quad \approx ( E q^{ -h_1 -h_2 +\lambda _1} + q^{3/2 -l_2 - \beta } ) ( E q^{1 -h_1 -h_2 +\lambda _1} + q^{7/2 -l_2 - \beta  } ) ( E q^{2 +\lambda _1 - h_1 - h_2} + q^{11/2 -l_2 - \beta  } )  \nonumber \\
& \qquad \qquad - q^{5 -l_1 - l_2 -\beta } ( E q^{ \lambda _1 - h_1 - h_2} + q^{3/2 -l_2 - \beta } )  \nonumber \\
& \quad \approx c_1 (E) \{ ( E q^{1 -h_1 -h_2 +\lambda _1} + q^{7/2 -l_2 - \beta  } ) ( E q^{2 +\lambda _1 - h_1 - h_2} + q^{11/2 -l_2 - \beta  } )- q^{5 -l_1 - l_2 -\beta } \} . \nonumber 
%
\end{align}
These are generalized as follows.
\begin{prop} \label{prop:reccnappr02}
Assume that $2N+1 +h_2 -l_2 -\beta <0 $ and $l_1 - l_2 -\beta   \not \in \{ -2N, -2N+2 ,\dots ,-4,-2 \}$.\\
Then we have
\begin{align}
 c_n (E) \approx c_{n-2} (E) & [( E q^{n-1 +\lambda _1 - h_1 - h_2} +q^{2n -1/2 - l_2 -\beta  } ) \label{eq:recapprnn-202} \\
& \quad  ( E q^{n-2 +\lambda _1 - h_1 - h_2} +q^{2n-2 -1/2 - l_2-\beta  } ) - q^{2n-1 -l_1 - l_2 -\beta  } ] \nonumber 
\end{align}
for $n=2,3,\dots ,N+1$.
\end{prop}
\begin{proof}
The formula for $n=2$ was shown in Eq.(\ref{eq:c2gen2}).
We show Eq.(\ref{eq:recapprnn-202}) for $n=m+1$ by assuming the case $n=m$.
It follows from the assumption that there are no cancellations of the leading terms on the right hand side of Eq.(\ref{eq:rec0102}) for $n=m +1$.
Hence it follows from Eq.(\ref{eq:rec0102}) that
\begin{align}
 c_{m+1} (E) \approx & \: c_{m} (E) ( E q^{m +\lambda _1 - h_1 - h_2} +q^{2m +2 - l_2 -\beta -1/2 } ) - c_{m-1} (E) q^{2m +1 -l_1 - l_2 -\beta } \label{eq:recapprnn-202pf} \\
 \approx & \:  c_{m-1} (E) ( E q^{m +\lambda _1 - h_1 - h_2} +q^{2m +2 - l_2 -\beta -1/2 } )  ( E q^{m-1 +\lambda _1 - h_1 - h_2} +q^{2m - l_2 -\beta -1/2 } ) \nonumber \\
& - c_{m-2} (E) q^{2m - 1 -l_1 - l_2 -\beta } ( E q^{m +\lambda _1 - h_1 - h_2} +q^{2m +2 - l_2 -\beta -1/2 } ) \nonumber \\
& - c_{m-1} (E) q^{2m +1 -l_1 - l_2 -\beta } . \nonumber 
%
\end{align}
It follows from Eq.(\ref{eq:rec0102}) for $n=m-1$ that 
\begin{align}
c_{m-1} (E) \approx & \: c_{m-2} (E)  ( E q^{m-2 +\lambda _1 - h_1 - h_2} +q^{2m -2 - l_2 -\beta -1/2 } )- c_{m-3} (E) q^{2m -3 -l_1 - l_2 -\beta }.
\label{eq:cm-1appr}
%
\end{align}
Since there are no cancellations of the leading terms on the right hand side of Eq.(\ref{eq:cm-1appr}), the term $ c_{m-1} (E) q^{2m +1 -l_1 - l_2 -\beta }$ is stronger than $ c_{m-2} (E) q^{2m +1 -l_1 - l_2 -\beta }   ( E q^{m-2 +\lambda _1 - h_1 - h_2} +q^{2n -2 - l_2 -\beta -1/2 } ) $ and it is stronger than $c_{m-2} (E) q^{2m - 1 -l_1 - l_2 -\beta } ( E q^{m +\lambda _1 - h_1 - h_2} +q^{2m +2 - l_2 -\beta -1/2 } ) $.
Hence we may reglect the term $c_{m-2} (E) q^{2m - 1 -l_1 - l_2 -\beta } ( E q^{m +\lambda _1 - h_1 - h_2} +q^{2m +2 - l_2 -\beta -1/2 } ) $ on the equivalence $\approx $ in Eq.(\ref{eq:recapprnn-202pf}) and we obtain (\ref{eq:recapprnn-202}) for $n=m+1$.
\end{proof}
By applying Proposition \ref{prop:reccnappr02} repeatedly, we obtain the approximation of the spectral polynomial $c_{N+1} (E) $ as follows.
\begin{thm}
We assume Eq.(\ref{eq:assump}), $2N+1 +h_2 -l_2 -\beta <0 $ and $l_1 - l_2 -\beta   \not \in \{ -2N, -2N+2 ,\dots ,-4,-2 \}$.
Set 
\begin{align}
p _n (E)& = ( E q^{n-1 +\lambda _1 - h_1 - h_2} +q^{2n -1/2 - l_2 -\beta  } ) \\
& \qquad \quad ( E q^{n-2 +\lambda _1 - h_1 - h_2} +q^{2n-2 -1/2 - l_2-\beta  } ) - q^{2n-1 -l_1 - l_2 -\beta  }  ,\nonumber \\
\tilde{c}_{N+1} (E) & = \left\{ \begin{array}{ll}
\displaystyle \prod_{n=1}^{(N+1)/2} p _{2n} (E), & \mbox{$N$ is odd,} \\
\displaystyle ( E q^{\lambda _1 - h_1 - h_2} +q^{3/2 - l_2 -\beta  } ) \prod_{n=1}^{N/2} p _{2n +1} (E), & \mbox{$N$ is even.} 
\end{array}
\right. \nonumber 
\end{align}
Then we have $c_{N+1} (E) \approx \tilde{c}_{N+1} (E)$.
\end{thm}
We investigate the zeros of $p_n (E)$ and $ \tilde{c}_{N+1} (E) $ as $q\to +0$.
If $ l_1 - l_2 -\beta <2 -2n$, then 
\begin{align}
p _n (E) \sim & q^{2n-3 +2(\lambda _1 - h_1 - h_2)} ( E +q^{n-3/2+\lambda_1 + l_{1}+\alpha _1 + \alpha _2 } ) ( E +q^{n-5/2+\lambda _1 +l_{1}+\alpha _1 + \alpha _2 } ) . 
%
\end{align}
If $2- 2 n < l_1 - l_2 -\beta  <3 -2n$, then 
\begin{align}
p _n (E) \sim & q^{2n-3 +2(\lambda _1 - h_1 - h_2)} (E+ q^{n-5/2+\lambda _1 +l_{1}+\alpha _1 + \alpha _2 } ) (E- q^{-n+5/2-\lambda _1 -l_{1}+h_{1}+h_{2} }). \label{eq:14n200} 
%
\end{align}
If $3 - 2 n < l_1 - l_2 -\beta  $, then the solutions of the quadratic equation $p_n(E)=0$ satisfy 
\begin{align}
E & = \frac{q^{n-5/2+\lambda _1 +l_{1}+\alpha _1 + \alpha _2 } \pm \sqrt{ q^{2(n-5/2+\lambda _1 +l_{1}+\alpha _1 + \alpha _2 )}+ 4 q^{\alpha _1 + \alpha _2  +h_1 + h_2 } }}{2} \nonumber \\
& \sim \pm q^{(\alpha _1 + \alpha _2  +h_1 + h_2 )/2} , 
\end{align} 
as $q \to +0$.

We describe the zeros of polynomial $\tilde{c}_{N+1} (E)  $ as $q\to +0$ by dividing into parts.\\
(i) If $ l_{1} - l_{2}  -\beta < -2N $, then 
\begin{align}
& E \sim  q^{j-3/2+\lambda _1 +l_{1}+\alpha _1 +\alpha _2 } \ \ \ \left( 1 \leq j \leq N + 1 \right).
\label{eq:asym2-1}
\end{align}
(ii-1) If $ l_{1} - l_{2}  -\beta > -1  $ and $N$ is odd, then
\begin{equation}
  E \sim - q^{( \alpha _1 +\alpha _2 +h_{1}+h_{2})/2}, \ q^{(\alpha _1 +\alpha _2 +h_{1}+h_{2})/2} \ \ \ \left( \mbox{$(N + 1)/2$-ple} \right).
\end{equation}
(ii-2) If $ l_{1} - l_{2} -\beta > -3 $, $ l_{1} - l_{2} -\beta \neq - 2 $ and $N$ is even, then
\begin{align}
  E  \sim & - q^{(\alpha _1 +\alpha _2 +h_{1}+h_{2})/2}, \ q^{(\alpha _1 +\alpha _2 +h_{1}+h_{2})/2} \ \ \ \left( \mbox{$N/2$-ple} \right),\\
          &  - q^{-1/2+\lambda _1 +l_{1}+\alpha _1 +\alpha _2 }. \nonumber 
\end{align}
(iii-1) If $ - 2N+ 4m < l_{1} - l_{2} -\beta < - 2N + 4m + 1 $ for some $m \in \Zint $ such that $0 \leq m \leq (N-1)/2$, then 
\begin{align}
  E &\sim  - q^{j-3/2+\lambda _1 +l_{1}+\alpha _1 +\alpha _2 }, \ \ \ \left( 1 \leq j \leq N-2m \right) , \\
    & \ \ \ \ q^{- N +2m +3/2-\lambda_1 -l_{1}+h_{1}+h_{2}}, \nonumber \\
    & \ \ \ \ - q^{(l_{3}+l_{4}+h_{1}+h_{2})/2}, \ q^{(l_{3}+l_{4}+h_{1}+h_{2})/2} \ \ \ \left( \mbox{$m$-ple} \right) .  \nonumber 
\end{align}
(iii-2) If $ - 2N+4m -3 <  l_{1} - l_{2} -\beta < - 2N +4m $ for some $m \in \Zint $ such that $1\leq m \leq (N-1)/2$ and $l_{1} - l_{2} -\beta \neq - 2N +4m -2 $, then 
\begin{align}
  E &\sim - q^{j-3/2+\lambda _1 +l_{1}+\alpha _1 +\alpha _2}, \ \ \ \left( 1 \leq j \leq N -2m +1 \right) , \\
    & \ \ \ \ - q^{(\alpha _1 +\alpha _2 +h_{1}+h_{2})/2}, \ q^{(\alpha _1 +\alpha _2 +h_{1}+h_{2})/2}, \ \ \ \left( \mbox{$m$-ple} \right) . \nonumber 
\end{align}
We can also discuss the zeros of the polynomials $c_{N+1} (E) $ as $q\to +0$ by comparing with the zeros of $ \tilde{c}_{N+1} (E) $ as the case $1 +h_2 -l_2 - \beta >0  $.
In particular, if $l_{1} - l_{2}  -\beta < -2N $, then the zeros of the polynomial $c_{N+1} (E) $ are written as $E_j(q)$ $(j=1,\dots N+1)$ such that $E_j(q) \sim -t_1 q^{j-3/2+\lambda _1 +l_{1}+\alpha _1 +\alpha _2 }$ for sufficiently small $q(>0)$ (see \cite{KST}), and it is compatible with the equivalence $c_{N+1} (E) \approx \tilde{c}_{N+1} (E) $ and Eq.(\ref{eq:asym2-1}).

\subsection{The case $-2N < 1 +h_2 -l_2 - \beta <0 $} \label{sec:sub3} $ $

We consider the polynomials $c_n(E)$ $(n=1,2,\dots ,N+1)$ for the case $-2N < 1 +h_2 -l_2 - \beta <0 $ and $1 +h_2 -l_2 - \beta \not \in \{-2N+2, \dots , -4,-2 \} $ with the condition in Eq.(\ref{eq:assump}).
In this case there exists $K \in \{ 1,2, \dots ,N \}$ such that $-2K< 1 +h_2 -l_2 - \beta < -2K+2$.
It follows from Eq.(\ref{eq:rec00}) that the polynomials $c_n(E)$ satisfy
\begin{align}
c_n (E) \sim & \: t_1^{-1} t_2^{-1}  (  E q^{n -1 -h_1 -h_2 +\lambda _1} + t_1 q^{ 2n -1/2 -l_2 - \beta  }) c_{n-1} (E)  - t_1^{-1} t_2^{-1} q^{2n-1 -l_1-l_2 -\beta } c_{n-2} (E) 
\label{eq:rec03-01}
%
\end{align}
for $ n=1,2,\dots ,K $ and
\begin{align}
c_n (E) \sim & \: t_1^{-1} t_2^{-1} ( E q^{n -1 -h_1 -h_2 +\lambda _1} + t_1 q^{1/2 -h_2 }) c_{n-1} (E) - t_1^{-1} t_2^{-1} q^{2n-1 -l_1-l_2 -\beta } c_{n-2} (E) 
\label{eq:rec03-02}
%
\end{align}
for $n=K +1 , K+2, \dots ,N+1$ under the assumption that there are no cancellations of the leading terms of the cofficients of $E^j$ ($j=0, 1, \dots  $) on the right hand sides.

We add the condition $h_2 - l_1 +1 >0 $ to avoid difficulty.
Then we obtain the following proposition.
\begin{prop} \label{prop:case3}
If $h_2 - l_1 +1 >0 $ and there exists $K \in \{ 1,2, \dots ,N \}$ such that $-2K< 1 +h_2 -l_2 - \beta < -2K+2$, then $c_n(E)$ satisfies
\begin{align}
& c_n (E) \sim \left\{ 
\begin{array}{l}
t_1^{-1} t_2^{-1} ( E q^{n -1 -h_1 -h_2 +\lambda _1} + q^{2n -1/2 -l_2 - \beta  } t_1 ) c_{n-1} (E) , \\
\qquad \qquad \qquad \qquad \qquad \qquad \qquad  n=1,2,\dots ,K, \\
t_1^{-1} t_2^{-1} ( E q^{n -1 -h_1 -h_2 +\lambda _1} + q^{1/2 -h_2 } t_1 ) c_{n-1} (E), \\
\qquad \qquad \qquad \qquad  n=K+1 ,K+2,\dots ,N+1.
\end{array}
\right.
\label{eq:rec03}
\end{align}
\end{prop}
\begin{proof}
Eq.(\ref{eq:rec03}) for $n=1$ follows from Eq.(\ref{eq:rec03-01}).
We show Eq.(\ref{eq:rec03}) for $n=m$ $(m=2,\dots ,K)$ by assuming Eq.(\ref{eq:rec03}) for $n=m-1$.
It follows from  Eq.(\ref{eq:rec03}) for $n=m-1$ that the right hand side of Eq.(\ref{eq:rec03-01}) for $n=m$ is written as
\begin{align}
& t_1^{-2} t_2^{-2} \{ ( E q^{m -1 -h_1 -h_2 +\lambda _1} + q^{2m -1/2 -l_2 - \beta  } t_1 )( E q^{m -2 -h_1 -h_2 +\lambda _1} + q^{2m -5/2 -l_2 - \beta  } t_1 )  \label{eq:rec03-1} \\
& - t_1 t_2 q^{2m-1 -l_1-l_2 -\beta }  \} c_{m-2} (E) .  \nonumber 
\end{align}
Then it follows from $ 4m -3 -2 l_2 - 2 \beta - (2m-1 -l_1-l_2 -\beta ) =  2m +l_1 -2 - l_2 - \beta < 2K +h_2 -1 - l_2- \beta  < 0 $ that we may ignore the term $t_1^{-1} t_2^{-1} q^{2m-1 -l_1-l_2 -\beta } c_{m-2} (E) $ in Eq.(\ref{eq:rec03-1}) and that in Eq.(\ref{eq:rec03-01}) for $n=m$.
Thus we have shown Eq.(\ref{eq:rec03}) for $n=m$.
Hence we obtain Eq.(\ref{eq:rec03}) for $n=1,\dots ,K$.

Next we show Eq.(\ref{eq:rec03}) for $n=K+1$.
It follows from  Eq.(\ref{eq:rec03}) for $n=K$ that the right hand side of Eq.(\ref{eq:rec03-02}) for $n=K+1$ is written as 
\begin{align}
& \: t_1^{-1} t_2^{-1} [  E q^{K -h_1 -h_2 +\lambda _1} + t_1 q^{1/2 -h_2 }] c_{K} (E) - t_1^{-1} t_2^{-1} q^{2K+1 -l_1-l_2 -\beta } c_{K-1} (E) \label{eq:rec03-3} \\
\sim & \: t_1^{-2} t_2^{-2} [ ( E q^{K -h_1 -h_2 +\lambda _1} + t_1 q^{1/2 -h_2 })( E q^{K -1 -h_1 -h_2 +\lambda _1} + q^{2K -1/2 -l_2 - \beta  } t_1 ) \nonumber \\
& - t_1 t_2 q^{2K+1 -l_1-l_2 -\beta } ] c_{K-1} (E) .  \nonumber 
%
\end{align}
It follows from $h_2 - l_1 +1 >0  $ that we may ignore the term $t_1^{-1} t_2^{-1} q^{2K+1 -l_1-l_2 -\beta } c_{K-1} (E)$ in Eq.(\ref{eq:rec03-3}) and that in Eq.(\ref{eq:rec03-02}) for $n=K+1$.
Hence we obtain Eq.(\ref{eq:rec03}) for $n=K+1$.
Let $m \in \{ K+1,\dots ,N\}$ and assume that Eq.(\ref{eq:rec03}) holds for $n= m $.
Then we have
\begin{align}
& t_1^{-2} t_2^{-2} \{( E q^{m -h_1 -h_2 +\lambda _1} + t_1 q^{1/2 -h_2 } ) c_{m} (E) - q^{2m +1 -l_1-l_2 -\beta } c_{m-1 } (E) \} \\
&  \sim  t_1^{-1} t_2^{-1} \{ (  E q^{m -h_1 -h_2 +\lambda _1} + t_1 q^{1/2 -h_2 })( E q^{m -h_1 -h_2 +\lambda _1} + t_1 q^{1/2 -h_2 } ) \nonumber \\
& \qquad - t_1 t_2 q^{2m +1 -l_1-l_2 -\beta } ]\} c_{m-1 } (E) . \nonumber
%
\end{align}
Since $1-2h_2 -( 2m +1 -l_1-l_2 -\beta )  < -2h_2 - 2K -2  +l_1 +l_2 + \beta <-h_2 - 2K -1 +l_2 + \beta <0 $, we may neglect the term $t_1^{-1} t_2^{-1} q^{2m +1 -l_1-l_2 -\beta } c_{m-1 } (E)  $ and we have  Eq.(\ref{eq:rec03}) for $n= m +1$.
\end{proof}
\begin{thm}
We assume Eq.(\ref{eq:assump}), $h_2 - l_1 +1 >0 $ and there exists $K \in \{ 1,2, \dots ,N \}$ such that $-2K< 1 +h_2 -l_2 - \beta < -2K+2$.\\
(i) The spectral polynomial $c_{N+1} (E)$ satisfies
\begin{align}
 c_{N+1} (E) & \sim (t_1 t_2 )^{-N-1} \prod _{n=1}^K ( E q^{n -1 -h_1 -h_2 +\lambda _1} + q^{2n -1/2 -l_2 - \beta  } t_1 ) \nonumber \\
& \qquad \prod _{n=K+1}^{N+1} ( E q^{n -1 -h_1 -h_2 +\lambda _1} + q^{1/2 -h_2 } t_1 ) \label{eq:cN+1case3} \\
& = (t_1 t_2 )^{-N-1} q^{(N/2 + \lambda _1 -h_1 -h_2)(N+1) } \prod _{n=1}^K ( E + q^{n-3/2 + \lambda _1 + l_1 + \alpha _1 +\alpha _2 } t_1 ) \nonumber \\
& \qquad \prod _{n=K+1}^{N+1} ( E + q^{-n + 3/2 +h_1 -\lambda _1 } t_1 ) \nonumber 
\end{align}
(ii) There exist solutions $E_j (q)$ $(j=1,2,\dots ,N+1)$ to the equation $c_{N+1} (E) =0 $ for sufficiently small $q$ such that
\begin{align}
& E _j(q)  \sim \left\{ 
\begin{array}{ll}
-q^{j-3/2 + \lambda _1 + l_1 + \alpha _1 +\alpha _2 } t_1 , & j=1,\dots ,K,\\
-q^{-j + 3/2 +h_1 -\lambda _1  } t_1 , & j=K+1, \dots , N+1.
\end{array}
\right.
\label{eq:Ekqcase3}
\end{align}
\end{thm}
\begin{proof}
We obtain (i) by applying Proposition \ref{prop:case3}.
Then it follows from the assumption $-2K< 1 +h_2 -l_2 - \beta < -2K+2$ that all of the zeros of the right hand side of Eq.(\ref{eq:cN+1case3}) are different by the ultradiscrete limit ($q \to +0 $).
Hence (ii) follows from the theorem in the appendix of \cite{KST}.
\end{proof}

\section{Concluding remarks} \label{sec:rem}

In this paper, we investigated roots of the spectral polynomial $c_{N+1} (E) $ as $q\to +0$ for three cases in the sections \ref{sec:sub1}, \ref{sec:sub2} and \ref{sec:sub3}.
Recall that the polynomial-type solution of the $q$-Heun equation exists, if the accessory parameter $E$ is a root of the spectral polynomial $c_{N+1} (E) $, but we did not investigate the polynomial-type solutions of the $q$-Heun equation in this paper, which we leave as a problem.

Another problem is to consider polynomial-type solutions to degenerations of the $q$-Heun equation.
Ultradiscrete limit would be applicable to those cases.
Note that degenerations of the $q$-Heun equation would be obtained similarly to the degenerations of Heun's differential equation (see \cite{Ron}).

\section*{Acknowledgements}
The third author was supported by JSPS KAKENHI Grant Number JP26400122.

\end{document}